\documentclass[10pt]{amsart} 
\usepackage{amsmath,amssymb,bm, color}

\usepackage{amsmath}
\usepackage{graphicx}

\usepackage{mathrsfs}
\usepackage{bbm}
\usepackage{bm}
\usepackage{amsfonts,amssymb}
\usepackage{multirow}
\usepackage{lineno}
\usepackage{color}

\numberwithin{equation}{section}
 \newtheorem{theorem}{Theorem}[section]
 \newtheorem{lemma}[theorem]{Lemma}

\def\3bar{{|\hspace{-.02in}|\hspace{-.02in}|}}
\def\E{{\mathcal{E}}}
\def\T{{\mathcal{T}}}

\def\cal#1{{\mathcal #1}}
\def\pT{{\partial T}}

\def\bv{{\mathbf{v}}}
\def\bn{{\mathbf{n}}}

\def\beta{{\boldsymbol{\eta}}}
\def\bvarphi{{\boldsymbol{\varphi}}}

\newtheorem{remark}{Remark}[section]
\newtheorem{algorithm}{Auto-Stabilized  Weak Galerkin Algorithm}[section]

\numberwithin{equation}{section}

\def\3bar{{|\hspace{-.02in}|\hspace{-.02in}|}}

 \def\cal#1{\mathcal{#1}}

\begin{document}

\title []{Auto-Stabilized Weak Galerkin Finite Element Methods on Polytopal Meshes without Convexity Constraints}

\author {Chunmei Wang}
\address{Department of Mathematics, University of Florida, Gainesville, FL 32611, USA. }
\email{chunmei.wang@ufl.edu}
\thanks{The research of Chunmei Wang was partially supported by National Science Foundation Grant DMS-2136380.}

\begin{abstract} 
This paper introduces an auto-stabilized weak Galerkin (WG) finite element method with a built-in stabilizer for Poisson equations. By utilizing bubble functions as a key analytical tool, our method extends to both convex and non-convex elements in finite element partitions, marking a significant advancement over existing stabilizer-free WG methods. It overcomes the restrictive conditions of previous approaches and is applicable in any dimension 
$d$, offering substantial advantages. The proposed method maintains a simple, symmetric, and positive definite structure. These benefits are evidenced by optimal order error estimates in both discrete 
$H^1$
  and 
$L^2$
  norms, highlighting the effectiveness and accuracy of our WG method for practical applications.
\end{abstract}


 
\maketitle

\section{Introduction}
In this paper, we aim to develop an auto-stabilized  weak Galerkin finite element method with a built-in stabilizer applicable to convex and non-convex polytopal meshes in finite element partitions, representing a significant improvement over existing stabilizer-free methods \cite{ye}.
To this aim and for simplicity, we consider the Poisson equation with homogeneous Dirichlet boundary condition that seeks an unknown function $u$ such that
\begin{equation}\label{model}
 \begin{split}
 -\Delta u=&f, \qquad\qquad \text{in}\quad \Omega,\\
 u=&0,\qquad\qquad \text{on}\quad \partial\Omega,
 \end{split}
 \end{equation}
where $\Omega\subset \mathbb R^d$ is an open bounded domain with Lipschitz boundary $\partial\Omega$. Note that the domain $\Omega$ considered in this paper can be any dimension $d$.

The variational formulation of the model problem \eqref{model} can be formulated  as follows: Find an unknown function $u\in H_0^1(\Omega)$ such that
\begin{equation}\label{weak}
(\nabla u, \nabla v)=(f, v), \qquad \forall v\in H_0^1(\Omega),
 \end{equation} 
  where $H_0^1(\Omega)=\{v\in H^1(\Omega): v=0 \ \text{on}\ \partial \Omega\}$.

 The weak Galerkin  finite element method represents a recent advancement in numerical techniques for solving  PDEs. This method reconstructs or approximates differential operators within a framework that parallels the theory of distributions for piecewise polynomials. Unlike traditional finite element methods, WG relaxes the typical regularity requirements for approximating functions through carefully designed stabilizers. Extensive studies have explored the applicability of the WG method across various model PDEs, supported by a comprehensive list of references \cite{wg1, wg2, wg3, wg4, wg5, wg6, wg7, wg8, wg9, wg10, wg11, wg12, wg13, wg14, wg15, wg16, wg17, wg18, wg19, wg20, wg21, itera, wy3655}, highlighting its potential as a robust tool in scientific computing. 
What distinguishes WG methods from other finite element methods is their reliance on weak derivatives and weak continuities to design numerical schemes based on the weak forms of underlying PDE problems. This structural flexibility makes WG methods particularly effective across a wide spectrum of PDEs, ensuring stability and accuracy in their approximations.

A notable advancement within the WG framework is the ``Primal-Dual Weak Galerkin (PDWG)" approach. This method addresses challenges that traditional numerical techniques struggle with \cite{pdwg1, pdwg2, pdwg3, pdwg4, pdwg5, pdwg6, pdwg7, pdwg8, pdwg9, pdwg10, pdwg11, pdwg12, pdwg13, pdwg14, pdwg15}. PDWG views numerical solutions as constrained minimizations of functionals, with constraints that mimic the weak formulation of PDEs using weak derivatives. This formulation results in an Euler-Lagrange equation that integrates both the primal variable and the dual variable (Lagrange multiplier), yielding a symmetric scheme. 
 
This paper introduces a straightforward formulation of the weak Galerkin finite element method that eliminates the need for stabilizers. Unlike existing stabilizer-free WG methods \cite{ye}, our approach is applicable to convex and non-convex polytopal meshes in any dimension 
$d$
 and supports flexible polynomial degrees. The key analytical tool enabling these advancements is the use of bubble functions. The trade-off involves utilizing higher-degree polynomials to compute the discrete weak gradient, which might be impractical for certain applications. However, our focus lies on the significant theoretical contributions to auto-stabilized WG methods with built-in stabilizers for non-convex elements in finite element partitions.

Our method preserves the size and global sparsity of the stiffness matrix, substantially simplifying programming complexity compared to traditional stabilizer-dependent WG methods. Theoretical analysis demonstrates that our WG approximations achieve optimal error estimates in both the discrete 
$H^1$
 and 
$L^2$
 norms. By offering an auto-stabilized WG method with built-in stabilizers  that maintains performance while reducing complexity, this paper makes a substantial contribution to the advancement of finite element methods on non-convex polytopal meshes.

 The contributions of our WG method are summarized as follows:
 \textbf{1. Handling Non-Convex Polytopal Meshes:} Unlike the method in \cite{ye}, which is restricted to convex elements in finite element partitions, our method effectively addresses non-convex polytopal meshes by utilizing bubble functions as a critical analysis tool. This significantly enhances its practical applicability.
\textbf{2. Flexibility with Bubble Functions:} Our method uses bubble functions without imposing restrictive conditions, in contrast to \cite{ye}. This flexibility facilitates broader generalization to various types of PDEs without the implementation complexities typically associated with such conditions proposed in \cite{ye}.
\textbf{3. Applicability in Higher Dimensions:}  While the method in \cite{ye} is limited to 2D or 3D settings, our approach extends to any dimension 
$d$, providing greater versatility for higher-dimensional problems.
\textbf{4. Flexible Polynomial Degrees:}  
Our approach supports varying polynomial degrees, offering increased flexibility in the discretization process.
Overall, these improvements result in a method that offers enhanced flexibility, broader applicability, and ease of implementation in diverse computational settings compared to existing stabilizer-free WG methods \cite{ye}.

While Sections 5-7 share some similarities with the work presented in \cite{ye}, our research introduces a more versatile WG scheme capable of handling non-convex polytopal elements in finite element partitions by employing bubble functions as a crucial analysis tool, as detailed above. To provide a comprehensive understanding of our contributions, we include an in-depth analysis of the error estimates in these sections. Although this analysis shares some aspects with \cite{ye}, it is essential for demonstrating the significant improvements and expanded applicability of our method. 

Although our algorithm shares some similarities with the one in \cite{ye}, it offers unique contributions, as detailed earlier. Our primary focus is on the algorithms and analytical advancements rather than empirical validation. The comprehensive numerical tests in \cite{ye} sufficiently demonstrate the efficacy of these algorithms, rendering their repetition unnecessary. Thus, our paper emphasizes challenging theoretical analysis, providing crucial insights for future development and applications.

This paper is structured as follows.  
Section 2 offers a brief review of the definition of the weak gradient and its discrete version. In Section 3, we introduce a simple weak Galerkin scheme that operates without the use of stabilizers. Section 4 is dedicated to proving the existence and uniqueness of the solution. In Section 5, we derive the error equation for the proposed weak Galerkin scheme. Section 6 focuses on deriving the error estimate for the numerical approximation in the energy norm. Finally, Section 7 establishes the error estimate for the numerical approximation in the  $L^2$ norm.  

The standard notations are adopted throughout this paper. Let $D$ be any open bounded domain with Lipschitz continuous boundary in $\mathbb{R}^d$. We use $(\cdot,\cdot)_{s,D}$, $|\cdot|_{s,D}$ and $\|\cdot\|_{s,D}$ to denote the inner product, semi-norm and norm in the Sobolev space $H^s(D)$ for any integer $s\geq0$, respectively. For simplicity, the subscript $D$ is  dropped from the notations of the inner product and norm when the domain $D$ is chosen as $D=\Omega$. For the case of $s=0$, the notations $(\cdot,\cdot)_{0,D}$, $|\cdot|_{0,D}$ and $\|\cdot\|_{0,D}$ are simplified as $(\cdot,\cdot)_D$, $|\cdot|_D$ and $\|\cdot\|_D$, respectively. 


\section{Weak Gradient and Discrete Weak Gradient}\label{Section:Hessian}
In this section, we shall briefly review the definition of weak gradient and its discrete version introduced in \cite{wy3655}.

Let $T$ be a polytopal element with boundary $\partial T$. A weak function on $T$  is defined as   $v=\{v_0, v_b\}$, where $v_0\in L^2(T)$, $v_b\in L^{2}(\partial T)$. The first component $v_0$ and the second component $v_b$ represent the value of $v$ in the interior of $T$ and on the boundary of $T$, respectively. Generally, $v_b$ is assumed to be independent of the trace of $v_0$. The special case occurs when $v_b= v_0|_{\partial T}$, where the function $v=\{v_0, v_b\}$ is uniquely determined by $v_0$ and can be simply denoted as $v=v_0$.
 
 Denote by $W(T)$ the space of all weak functions on $T$; i.e.,
 \begin{equation}\label{2.1}
 W(T)=\{v=\{v_0,v_b\}: v_0\in L^2(T), v_b\in L^{2}(\partial
 T)\}.
\end{equation}
 
 The weak gradient, denoted by $\nabla_{w}$, is a linear
 operator from $W(T)$ to the dual space of $[H^{1}(T)]^d$. For any
 $v\in W(T)$, the weak gradient  $\nabla_w v$ is defined as a bounded linear functional on $[H^{1}(T)]^d$
such that
 \begin{equation}\label{2.3}
  (\nabla _{ w}v, \bvarphi)_T=-(v_0,\nabla\cdot \bvarphi)_T+
  \langle v_b, \bvarphi\cdot \bn \rangle_{\partial T},\quad \forall \bvarphi\in [H^{1}(T)]^d,
  \end{equation}
 where $ \bn$ is an unit outward normal direction to $\partial T$.
 
 For any non-negative integer $r$, let $P_r(T)$ be the space of
 polynomials on $T$ with total degree at most 
 $r$. A discrete weak
 gradient on $T$, denoted by $\nabla_{w, r, T}$, is a linear operator
 from $W(T)$ to $[P_r(T)]^d$. For any $v\in W(T)$,
 $\nabla_{w, r, T}v$ is the unique polynomial vector in $[P_r(T)]^d$ satisfying
 \begin{equation}\label{2.4}
  (\nabla_{w,r,T} v, \bvarphi)_T=-(v_0,\nabla\cdot \bvarphi)_T+
  \langle v_b,  \bvarphi\cdot \bn \rangle_{\partial T},\quad \forall \bvarphi \in [P_r(T)]^d.
  \end{equation}
 For a smooth $v_0\in
 H^1(T)$,  applying the usual integration by parts to the first
 term on the right-hand side of (\ref{2.4})  gives
 \begin{equation}\label{2.4new}
  (\nabla_{w,r,T} v, \bvarphi)_T=(\nabla v_0, \bvarphi)_T+
  \langle v_b-v_0,  \bvarphi\cdot \bn \rangle_{\partial T},\quad \forall \bvarphi \in [P_r(T)]^d.
  \end{equation}

\section{Auto-Stabilized Weak Galerkin Algorithms}\label{Section:WGFEM}
 Let ${\cal T}_h$ be a finite element partition of the domain
 $\Omega\subset \mathbb R^d$ into polytopes. Assume that ${\cal
 T}_h$ is shape regular   \cite{wy3655}.
 Denote by ${\mathcal E}_h$ the set of all edges/faces  in
 ${\cal T}_h$ and ${\mathcal E}_h^0={\mathcal E}_h \setminus
 \partial\Omega$ the set of all interior edges/faces. Denote
 by $h_T$ the diameter of $T\in {\cal T}_h$ and $h=\max_{T\in {\cal
 T}_h}h_T$ the meshsize of the finite element partition ${\cal
 T}_h$.

 Let $k\geq 0$ and $q \geq 0$ be two nonnegative integers. We assume that $k\geq q$. For any element $T\in\T_h$, define a local
 weak finite element space; i.e.,
 \begin{equation}
 V(k, q, T)=\{\{v_0,v_b\}: v_0\in P_k(T),v_b\in P_q(e), e\subset \partial T\}.   
 \end{equation}
By patching $V(k, q,T)$ over all the elements $T\in {\cal T}_h$ through
 a common value $v_b$ on the interior interface $\E_h^0$,
 we obtain a global weak finite element space; i.e.,
 $$
 V_h=\big\{\{v_0,v_b\}:\ \{v_0,v_b\}|_T\in V(k, q, T),
 \forall T\in {\cal T}_h \big\}.
 $$
Denote by $ V_h^0$ the subspace of $V_h$ with vanishing boundary value on $\partial\Omega$; i.e.,
$$
V_h^0=\{v\in V_h: v_b=0 \ \text{on}\ \partial\Omega\}.
$$

For simplicity of notation and without confusion, for any $v\in V_h$, we denote the discrete weak gradient 
$\nabla_{w, r, T} v$ simply as
  $\nabla_{w}v$,  computed on each element $T$ using
(\ref{2.4}); that is,
$$
(\nabla_{w} v)|_T= \nabla_{w, r, T}(v |_T), \qquad \forall T\in \T_h.
$$ 

The simple WG numerical scheme without the use of stabilizers for the Possion equation  (\ref{model}) is formulated as follows.
\begin{algorithm}\label{PDWG1}
Find $ u_h=\{u_0, u_b\} \in V_h^0$, such that  
\begin{equation}\label{WG}
(\nabla_w u_h, \nabla_w v)=(f, v_0), \qquad\forall v=\{v_0, v_b\}\in V_h^0,
\end{equation}
where
$$
(\nabla_w u_h, \nabla_w v)=\sum_{T\in {\cal T}_h} (\nabla_w u_h, \nabla_w v)_T,
$$ 
$$
(f, v_0)=\sum_{T\in {\cal T}_h}(f, v_0)_T.
$$
\end{algorithm}

\section{Solution Existence and Uniqueness} 
 
Recall that ${\cal T}_h$ is a shape-regular finite element partition of the domain $\Omega$. Therefore,  for any $T\in {\cal T}_h$ and $\phi\in H^1(T)$,
 the following trace inequality holds true \cite{wy3655}; i.e.,
\begin{equation}\label{tracein}
 \|\phi\|^2_{\partial T} \leq C(h_T^{-1}\|\phi\|_T^2+h_T \|\nabla \phi\|_T^2).
\end{equation}
If $\phi$ is a polynomial on the element $T\in {\cal T}_h$,  the following trace inequality holds true \cite{wy3655}; i.e.,
\begin{equation}\label{trace}
\|\phi\|^2_{\partial T} \leq Ch_T^{-1}\|\phi\|_T^2.
\end{equation}

For any $v=\{v_0, v_b\}\in V_h$, we define
the following discrete energy norm \begin{equation}\label{3norm}
\3bar v\3bar=\Big( \sum_{T\in {\cal T}_h} (\nabla_w v, \nabla_w v)_T\Big)^{\frac{1}{2}},
\end{equation}
and the following discrete $H^1$ semi-norm 
\begin{equation}\label{disnorm}
\|v\|_{1, h}=\Big( \sum_{T\in {\cal T}_h} \|\nabla v_0\|_T^2+h_T^{-1}\|v_0-v_b\|_{\partial T}^2\Big)^{\frac{1}{2}}.
\end{equation}
\begin{lemma}\label{norm1}
 For $v=\{v_0, v_b\}\in V_h$, there exists a constant $C$ such that
 $$
 \|\nabla v_0\|_T\leq C\|\nabla_w v\|_T.
 $$
\end{lemma}
\begin{proof}  Let  $T\in {\cal T}_h$ be a polytopal element with $N$ edges/faces denoted by $e_1, \cdots, e_N$. It is important to emphasis that the polytopal element $T$  can be non-convex. On each edge/face $e_i$, we construct   a linear equation  $l_i(x)$ such that  $l_i(x)=0$ on edge/face $e_i$ as follows: 
$$l_i(x)=\frac{1}{h_T}\overrightarrow{AX}\cdot \bn_i, $$  where  $A=(A_1, \cdots, A_{d-1})$ is a given point on the edge/face $e_i$,  $X=(x_1, \cdots, x_{d-1})$ is any point on the edge/face $e_i$, $\bn_i$ is the normal direction to the edge/face $e_i$, and $h_T$ is the size of the element $T$. 

The bubble function of  the element  $T$ can be  defined as 
 $$
 \Phi_B =l^2_1(x)l^2_2(x)\cdots l^2_N(x) \in P_{2N}(T).
 $$ 
 It is straightforward to verify that  $\Phi_B=0$ on the boundary $\partial T$.    The function 
  $\Phi_B$  can be scaled such that $\Phi_B(M)=1$ where   $M$  represents the barycenter of the element $T$. Additionally,  there exists a sub-domain $\hat{T}\subset T$ such that $\Phi_B\geq \rho_0$ for some constant $\rho_0>0$.

For $v=\{v_0, v_b\}\in V_h$, letting $r=2N+k-1$ and
$\bvarphi=\Phi_B \nabla v_0\in [P_r(T)]^d$ in \eqref{2.4new}, we obtain
\begin{equation}\label{t1}
\begin{split}
(\nabla_w v, \Phi_B \nabla v_0)_T&=(\nabla v_0, \Phi_B \nabla v_0)_T+\langle v_b-v_0,  \Phi_B \nabla v_0 \cdot \bn\rangle_{\partial T}\\&=(\nabla v_0, \Phi_B \nabla v_0)_T,
\end{split}
\end{equation}
where we used $\Phi_B=0$ on $\partial T$.

From the domain inverse inequality \cite{wy3655},  there exists a constant $C$ such that 
\begin{equation}\label{t2}
(\nabla v_0, \Phi_B \nabla v_0)_T \geq C (\nabla v_0, \nabla v_0)_T.
\end{equation} 
From Cauchy-Schwarz inequality and \eqref{t1}-\eqref{t2}, we have
 $$
 (\nabla v_0, \nabla v_0)_T\leq C (\nabla_w v, \Phi_B \nabla v_0)_T  \leq C  \|\nabla_w v\|_T \|\Phi_B \nabla v_0\|_T  \leq C
\|\nabla_w v\|_T \|\nabla v_0\|_T,
 $$
which gives
 $$
 \|\nabla v_0\|_T\leq C\|\nabla_w v\|_T.
 $$

This completes the proof of the lemma.
\end{proof}
 
\begin{remark}
   If the polytopal element $T$  is convex, 
   the bubble function of  the element  $T$ in Lemma \ref{norm1}  can be  simplified to
 $$
 \Phi_B =l_1(x)l_2(x)\cdots l_N(x).
 $$ 
It can be verified that there exists a sub-domain $\hat{T}\subset T$,  such that
 $ \Phi_B\geq\rho_0$  for some constant $\rho_0>0$,  and $\Phi_B=0$ on the boundary $\partial T$.   Lemma \ref{norm1}   can be proved in the same manner using this simplified construction. In this case, we take $r=N+k-1$.  
\end{remark}

Recall that $T$ is a $d$-dimensional polytopal element and  $e_i$ is a $d-1$-dimensional edge/face  of $T$. 
We construct an edge/face-based bubble function   $$\varphi_{e_i}= \Pi_{k=1, \cdots, N, k\neq i}l_k^2(x).$$ It can be verified that  (1) $\varphi_{e_i}=0$ on the edge/face $e_k$ for $k \neq i$, (2) there exists a subdomain $\widehat{e_i}\subset e_i$ such that $\varphi_{e_i}\geq \rho_1$ for some constant $\rho_1>0$.

\begin{lemma}\label{phi}
     For $v=\{v_0, v_b\}\in V_h$, let $\bvarphi=(v_b-v_0) \bn\varphi_{e_i}$, where $\bn$ is the unit outward normal direction to the edge/face  $e_i$. The following inequality holds:
\begin{equation}
  \|\bvarphi\|_T ^2 \leq Ch_T \int_{e_i}((v_b-v_0)\bn)^2ds.
\end{equation}
\end{lemma}
\begin{proof}
 We first extend $v_b$, initially defined on the $(d-1)$-dimensional edge/face  $e_i$, to the entire d-dimensional polytopal element $T$  using  the following formula:
$$
 v_b (X)= v_b(Proj_{e_i} (X)),
$$
where $X=(x_1,\cdots,x_d)$ is any point in the  element $T$, $Proj_{e_i} (X)$ denotes the orthogonal projection of the point $X$ onto  the hyperplane $H\subset\mathbb R^d$  containing the edge/face  $e_i$. 
When the projection $Proj_{e_i} (X)$ is not on the edge/face $e_i$, $v_b(Proj_{e_i} (X))$ is defined to be  the extension of $v_b$ from $e_i$ to the hyperplane $H$.

We claim that $v_b$ remains  a polynomial defined on the element $T$ after the extension.  

Let the hyperplane $H$ containing the edge/face  $e_i$ be defined by $d-1$ linearly independent vectors $\beta_1, \cdots, \beta_{d-1}$ originating from a point $A$ on $e_i$. Any point $P$ on the hyperplane $H$ can be parametrized as
$$
P(t_1, \cdots, t_{d-1})=A+t_1\beta_1+\cdots+t_{d-1}\beta_{d-1},
$$
where $t_1, \cdots, t_{d-1}$ are parameters.

Note that $v_b(P(t_1, \cdots, t_{d-1}))$ is a polynomial of degree $q$ defined on  the edge/face  $e_i$. It can be expressed as:
$$
v_b(P(t_1, \cdots, t_{d-1}))=\sum_{|\alpha|\leq q}c_{\alpha}\textbf{t}^{\alpha},
$$
where $\textbf{t}^{\alpha}=t_1^{\alpha_1}\cdots t_{d-1}^{\alpha_{d-1}}$ and $\alpha=(\alpha_1, \cdots, \alpha_{d-1})$  is a multi-index.

For any point $X=(x_1, \cdots, x_d)$ in the element $T$, the projection of the point $X$  onto  the hyperplane $H\subset\mathbb R^d$  containing the edge/face  $e_i$ is the point on the hyperplane $H$ that minimizes the distance to  $X$. Mathematically, this projection $Proj_{e_i} (X)$ is an affine transformation which can be expressed as 
$$
Proj_{e_i} (X)=A+\sum_{i=1}^{d-1} t_i(X)\beta_i,
$$
where $t_i(X)$ are the projection coefficients, and $A$ is the origin point on $e_i$. The coefficients $t_i(X)$ are determined  by solving the orthogonality condition:
$$
(X-Proj_{e_i} (X))\cdot \beta_j=0, \forall j=1, \cdots, d-1.
$$
This results in a system of linear equations in $t_1(X)$, $\cdots$, $t_{d-1}(X)$, which  can be solved to yield:
$$
t_i(X)= \text{linear function of} \  X.
$$
Hence, the projection $Proj_{e_i} (X)$ is an affine linear function  of $X$.

We extend the polynomial $v_b$ from the edge/face  $e_i$ to the entire element $T$ by defining
$$
v_b(X)=v_b(Proj_{e_i} (X))=\sum_{|\alpha|\leq q}c_{\alpha}\textbf{t}(X)^{\alpha},
$$
where $\textbf{t}(X)^{\alpha}=t_1(X)^{\alpha_1}\cdots t_{d-1}(X)^{\alpha_{d-1}}$. Since $t_i(X)$ are linear functions of $X$, each term $\textbf{t}(X)^{\alpha}$ is a polynomial in $X=(x_1, \cdots, x_d)$.
Thus, $v_b(X)$ is a polynomial in the $d$-dimensional coordinates $X=(x_1, \cdots, x_d)$.

 Secondly, let $v_{trace}$ denote the trace of $v_0$ on the edge/face  $e_i$. We extend $v_{trace}$   to the entire element $T$  using  the following formula:
$$
 v_{trace} (X)= v_{trace}(Proj_{e_i} (X)),
$$
where $X$ is any point in the element $T$, $Proj_{e_i} (X)$ denotes the projection of the point $X$ onto the hyperplane $H$ containing the edge/face  $e_i$. When the projection $Proj_{e_i} (X)$ is not on the edge/face $e_i$, $v_{trace}(Proj_{e_i} (X))$ is defined to be  the extension of $v_{trace}$ from $e_i$ to the hyperplane $H$.
Similar to the case for $v_b$, $v_{trace}$ remains a polynomial after this extension. 

Let $\bvarphi=(v_b-v_0) \bn\varphi_{e_i}$. We have
\begin{equation*}
    \begin{split}
\|\bvarphi\|^2_T  =
\int_T \bvarphi^2dT =  &\int_T ((v_b-v_{trace})  \bn\varphi_{e_i})^2dT\\
\leq &Ch_T \int_{e_i} ((v_b-v_{trace})  \bn\varphi_{e_i})^2ds\\
\leq &Ch_T \int_{e_i} ((v_b-v_0)\bn)^2ds,
    \end{split}
\end{equation*} 
where we used the facts that  (1) $\varphi_{e_i}=0$ on the edge/face $e_k$ for $k \neq i$, (2) there exists a subdomain $\widehat{e_i}\subset e_i$ such that $\varphi_{e_i}\geq \rho_1$ for some constant $\rho_1>0$, and applied the properties of the projection.

 This completes the proof of the lemma.

\end{proof}

\begin{remark}
Consider any $d$-dimensional polytopal element $T$. 
  There exists a hyperplane $H\subset R^d$  such that a finite number $l$ of distinct $(d-1)$-dimensional edges/faces containing $e_{i}$ are completely contained within $H$. 
 For simplicity,   we consider the case of $l=2$: there exists a   hyperplane $H\subset R^d$    such that distinct $(d-1)$-dimensional edges/faces $e_i, e_m$ are completely contained within $H$.
 The edge/face-based bubble function  is defined as
$$\varphi_{e_i}= \Pi_{k=1, \cdots, N, k\neq i, m}l_k^2(x).$$ 
It is easy to check $\varphi_{e_i} = 0$ on edges/faces $e_k$ for $k \neq i, m$.


When the above case occurs, Lemma \ref{phi} needs to be revised as follows to accommodate the case. The last part of the proof of Lemma \ref{phi} is revised as follows.

Let $\sigma$ be a polynomial function satisfying $\sigma \geq \alpha$ on $e_i$ for some constant $\alpha > 0$ and $|\sigma| \leq \epsilon$ on $e_m$ for a sufficiently small $\epsilon > 0$. Let $\bvarphi = (v_b - v_0) \bn \varphi_{e_i} \sigma$. We have
\begin{equation*}
    \begin{split}
        \|\bvarphi\|^2_T =& \int_T \bvarphi^2 \, dT \\
        =& \int_T \left( (v_b - v_{trace}) \bn \varphi_{e_i} \sigma \right)^2 \, dT \\
        \leq & C h_T \int_{e_i} \left( (v_b - v_{trace})  \bn \varphi_{e_i} \sigma \right)^2  ds \\
             & + C h_T \int_{e_m}  \left( (\tilde{v}_b - \tilde{v}_{trace}) \bn \varphi_{e_i} \sigma \right)^2  ds \\
        \leq & C h_T \|(v_b - v_0) \bn    \|^2_{e_i} 
              + C h_T \epsilon^2 \|(\tilde{v}_b - \tilde{v}_0) \bn \|^2_{e_m} \\
        \leq & C h_T \|(v_b - v_0) \bn  \|^2_{e_i} 
              + C h_T \epsilon^2 \|(\tilde{v}_b - \tilde{v}_0) \bn\|^2_{e_i \cup e_m} \\
        \leq & C h_T \|(v_b - v_0) \bn  \|^2_{e_i} 
              + C h_T \epsilon^2 \|(v_b - v_0) \bn\|^2_{e_i  } \\
        \leq & C h_T \|(v_b - v_0) \bn\|^2_{e_i},
    \end{split}
\end{equation*} 
where  $\tilde{v}_b$ and $\tilde{v}_{trace}$ are the extensions of $v_b$ and $v_{trace}$ from $e_i$ to $e_m$. 
Here, we used the following facts:
(1) $\varphi_{e_i} = 0$ on edges/faces $e_k$ for $k \neq i, m$.
(2) There exists a subdomain $\widehat{e}_i \subset e_i$ such that $\varphi_{e_i} \geq \rho_1$ for some constant $\rho_1 > 0$. 
(3)  The domain inverse inequality.
(4)  The properties of the projection operators.
\end{remark}

\begin{lemma}\label{normeqva}   There exists two positive constants $C_1$ and $C_2$ such that for any $v=\{v_0, v_b\} \in V_h$, we have
 \begin{equation}\label{normeq}
 C_1\|v\|_{1, h}\leq \3bar v\3bar  \leq C_2\|v\|_{1, h}.
\end{equation}
\end{lemma}

\begin{proof}   
 Note that the polytopal element $T$ can be non-convex.
Recall that an edge/face-based bubble function  is defined as  $$\varphi_{e_i}= \Pi_{k=1, \cdots, N, k\neq i}l_k^2(x).$$   


We first extend $v_b$ from the edge/face $e_i$ to the element $T$. 
Next, let $v_{trace}$ denote the trace of $v_0$ on the edge/face $e_i$ and extend $v_{trace}$ to the element $T$. For simplicity, we continue to denote these extensions as $v_b$ and $v_0$. Details of the extensions can be found in Lemma \ref{phi}.

Choosing $\bvarphi=(v_b-v_0) \bn\varphi_{e_i}$ in \eqref{2.4new}, gives
\begin{equation}\label{t3} 
  (\nabla_{w} v, \bvarphi)_T=(\nabla v_0, \bvarphi)_T+
  {  \langle v_b-v_0,  \bvarphi\cdot \bn \rangle_{\partial T}} =(\nabla v_0, \bvarphi)_T+ \int_{e_i}|v_b-v_0|^2 \varphi_{e_i}ds,
  \end{equation} 
   where we used (1) $\varphi_{e_i}=0$ on the edge/face $e_k$ for $k \neq i$, 
(2) there exists a subdomain $\widehat{e_i}\subset e_i$ such that $\varphi_{e_i}\geq \rho_1$ for some constant $\rho_1>0$.

   Using Cauchy-Schwarz inequality, \eqref{t3}, the domain inverse inequality \cite{wy3655}, and  Lemma \ref{phi} gives
\begin{equation*}
\begin{split}
 \int_{e_i}|v_b-v_0|^2  ds\leq &C  \int_{e_i}|v_b-v_0|^2  \varphi_{e_i}ds \\
 \leq & C(\|\nabla_w v\|_T+\|\nabla v_0\|_T)\| \bvarphi\|_T\\
 \leq & {Ch_T^{\frac{1}{2}} (\|\nabla_w v\|_T+\|\nabla v_0\|_T) (\int_{e_i}((v_0-v_b)\bn)^2ds)^{\frac{1}{2}}},
 \end{split}
\end{equation*}
which, from Lemma \ref{norm1}, gives 
$$
 h_T^{-1}\int_{e_i}|v_b-v_0|^2  ds \leq C  (\|\nabla_w v\|^2_T+\|\nabla v_0\|^2_T)\leq C\|\nabla_w v\|^2_T.
$$
This, together with Lemma \ref{norm1},  \eqref{3norm} and \eqref{disnorm}, gives
$$
 C_1\|v\|_{1, h}\leq \3bar v\3bar.
$$

Next, from \eqref{2.4new}, Cauchy-Schwarz inequality and  the trace inequality \eqref{trace}, we have
$$
 \Big|(\nabla_{w} v, \bvarphi)_T\Big| \leq \|\nabla v_0\|_T \|  \bvarphi\|_T+
Ch_T^{-\frac{1}{2}}\|v_b-v_0\|_{\partial T} \| \bvarphi\|_{T},
$$
which yields
$$
\| \nabla_{w} v\|_T^2\leq C( \|\nabla v_0\|^2_T  +
 h_T^{-1}\|v_b-v_0\|^2_{\partial T}),
$$
 and further gives $$ \3bar v\3bar  \leq C_2\|v\|_{1, h}.$$

 This completes the proof of the lemma.
 \end{proof}

  \begin{remark}
   If the polytopal element $T$  is convex, 
  the edge/face-based bubble function in Lemma \ref{normeqva}  can be  simplified to
$$\varphi_{e_i}= \Pi_{k=1, \cdots, N, k\neq i}l_k(x).$$
It can be verified that (1)  $\varphi_{e_i}=0$ on the edge/face $e_k$ for $k \neq i$, (2) there exists a subdomain $\widehat{e_i}\subset e_i$ such that $\varphi_{e_i}\geq \rho_1$ for some constant $\rho_1>0$. Lemma \ref{normeqva}   can be proved in the same manner using this simplified construction.  
\end{remark}

 \begin{remark}
Consider any $d$-dimensional polytopal element $T$. 
  There exists a hyperplane $H\subset R^d$  such that a finite number $l$ of distinct $(d-1)$-dimensional edges/faces containing $e_{i}$ are completely contained within $H$. 
 For simplicity,   we consider the case of $l=2$: there exists a   hyperplane $H\subset R^d$    such that distinct $(d-1)$-dimensional edges/faces $e_i, e_m$ are completely contained within $H$.
 The edge/face-based bubble function  is defined as
$$\varphi_{e_i}= \Pi_{k=1, \cdots, N, k\neq i, m}l_k^2(x).$$ 
It is easy to check $\varphi_{e_i} = 0$ on edges/faces $e_k$ for $k \neq i, m$.

When the above case occurs, Lemma \ref{normeqva} needs to be revised as follows to accommodate the case. The related  part of the proof of Lemma \ref{normeqva} is revised as follows.

 Let $\sigma$ be a polynomial function satisfying $\sigma \geq \alpha$ on $e_i$ for some constant $\alpha>0$ and $|\sigma| \leq \epsilon$ on $e_m$ for a sufficiently small $\epsilon>0$. Choosing $\bvarphi=(v_b-v_0) \bn\varphi_{e_i}\sigma$ in \eqref{2.4new}, gives   \begin{equation}\label{t3-1}  
     \begin{split}
    &  (\nabla_{w} v, \bvarphi)_T\\ =&(\nabla v_0, \bvarphi)_T+
  {  \langle v_b-v_0,  \bvarphi\cdot \bn \rangle_{\partial T}} \\ =&(\nabla v_0, \bvarphi)_T+ \int_{e_{i}}|v_b-v_0|^2  \varphi_{e_i}\sigma ds+\int_{e_{m}}(v_b-v_0) (\tilde{v}_b-\tilde{v}_0)  \varphi_{e_i}\sigma ds\\
  \geq &(\nabla v_0, \bvarphi)_T+ \int_{e_{i}}|v_b-v_0|^2  \varphi_{e_i}\alpha  ds+\int_{e_{m}}(v_b-v_0)  (\tilde{v}_b-\tilde{v}_0)  \varphi_{e_i}\sigma ds, 
     \end{split}
  \end{equation}
   where 
  $\tilde{v}_b$ and $\tilde{v}_0$ are the extensions of $v_b$ and $v_0$  from  $e_i$ to $e_m$.

Using Cauchy-Schwarz inequality, \eqref{t3-1}, the domain inverse inequality \cite{wy3655}, and  Lemma \ref{phi} gives
\begin{equation}\label{s1}
\begin{split}
 &\int_{e_i}|v_b-v_0|^2  ds\\\leq &C  \int_{e_i}|v_b-v_0|^2  \varphi_{e_i}\alpha ds \\
 \leq & C(\|\nabla_w v\|_T+\|\nabla v_0\|_T)\| \bvarphi\|_T +  \Big|\int_{e_{m}}(v_b-v_0)  (\tilde{v}_b-\tilde{v}_0)  \varphi_{e_i}\sigma ds\Big|\\
 \leq & Ch_T^{\frac{1}{2}} (\|\nabla_w v\|_T+\|\nabla v_0\|_T)  \|v_0-v_b\|_{e_i} +  C\epsilon \|v_b-v_0\|_{e_m} \| \tilde{v}_b-\tilde{v}_0\|_{e_m}\\
 \leq & Ch_T^{\frac{1}{2}} (\|\nabla_w v\|_T+\|\nabla v_0\|_T)  \|v_0-v_b\|_{e_i} +  C\epsilon \|v_b-v_0\|_{e_m} \| \tilde{v}_b-\tilde{v}_0\|_{e_m\cup e_i}\\
  \leq & Ch_T^{\frac{1}{2}} (\|\nabla_w v\|_T+\|\nabla v_0\|_T)  \|v_0-v_b\|_{e_i} +  C\epsilon \|v_b-v_0\|_{e_m} \| v_b-v_0\|_{e_i}.
 \end{split}
\end{equation}
This gives 
$$
\| v_b-v_0\|_{e_i} \leq Ch_T^{\frac{1}{2}} (\|\nabla_w v\|_T+\|\nabla v_0\|_T)   +  C\epsilon \|v_b-v_0\|_{e_m}.
$$

 Analogous to \eqref{s1}, we have
\begin{equation}\label{s2}
    \| v_b-v_0\|_{e_m} \leq Ch_T^{\frac{1}{2}} (\|\nabla_w v\|_T+\|\nabla v_0\|_T)   +  C\epsilon \|v_b-v_0\|_{e_i}.
\end{equation}

Substituting \eqref{s2} into \eqref{s1} gives
$$
\| v_b-v_0\|_{e_i} \leq Ch_T^{\frac{1}{2}} (1+\epsilon)(\|\nabla_w v\|_T+\|\nabla v_0\|_T)   + C \epsilon^2 \|v_b-v_0\|_{e_i}. 
$$
This gives 
$$
(1-\epsilon^2)\| v_b-v_0\|_{e_i} \leq Ch_T^{\frac{1}{2}} (1+\epsilon)(\|\nabla_w v\|_T+\|\nabla v_0\|_T). 
$$
Thus, we have
$$
\| v_b-v_0\|_{e_i} \leq Ch_T^{\frac{1}{2}} (\|\nabla_w v\|_T+\|\nabla v_0\|_T). 
$$
 \end{remark}
\begin{theorem}
The  WG Algorithm  \ref{PDWG1} has  a unique solution. 
\end{theorem}
\begin{proof}
Assume $u_h^{(1)}\in V_h^0$ and $u_h^{(2)}\in V_h^0$ are two different solutions of the WG Algorithm  \ref{PDWG1}. Then, $\eta_h= u_h^{(1)}-u_h^{(2)}\in V_h^0$ satisfies
$$
(\nabla_w \eta_h, \nabla_w v)=0, \qquad \forall v\in V_h^0.
$$
Letting $v=\eta_h$  gives $\3bar \eta_h\3bar=0$. From \eqref{normeq} we have $\|\eta_h\|_{1,h}=0$, which yields
$\nabla \eta_0=0$ on each $T$ and $\eta_0=\eta_b$ on each $\partial T$. Using the fact that $\nabla \eta_0=0$ on each $T$ gives $\eta_0=C$ on each $T$. This, together with $\eta_0=\eta_b$ on each $\partial T$ and  $\eta_b=0$ on $\partial \Omega$,   gives $\eta_0\equiv 0$ and  further $\eta_b\equiv 0$  and $\eta_h\equiv 0$ in the domain $\Omega$. Therefore, we have $u_h^{(1)}\equiv u_h^{(2)}$. This completes the proof of this theorem.
\end{proof}

\section{Error Equations}
On each element $T\in\T_h$, let $Q_0$ be the $L^2$ projection onto $P_k(T)$. On each  edge/face  $e\subset\partial T$, let $Q_b$ be the $L^2$ projection operator onto $P_{q}(e)$. For any $w\in H^1(\Omega)$, denote by $Q_h w$ the $L^2$ projection into the weak finite element space $V_h$ such that
$$
(Q_hw)|_T:=\{Q_0(w|_T),Q_b(w|_{\pT})\},\qquad \forall T\in\T_h.
$$
Denote by $Q_r$ the $L^2$ projection operator onto the finite element space of piecewise polynomials of degree $r$.

\begin{lemma}\label{Lemma5.1}   The following property holds true; i.e.,
\begin{equation}\label{pro}
\nabla_{w}u =Q_r\nabla u, \qquad \forall u\in H^1(T).
\end{equation}
\end{lemma}

\begin{proof} For any $u\in H^1(T)$,  using \eqref{2.4new} gives 
 \begin{equation*} 
 \begin{split}
  (\nabla_{w} u,\bvarphi)_T=&(\nabla u, \bvarphi)_T+
  \langle u|_{\partial T}-u|_T, \bvarphi \cdot \bn \rangle_{\partial T}\\=&(\nabla u, \bvarphi)_T\\
  =&(Q_r\nabla u, \bvarphi)_T, 
  \end{split}
  \end{equation*}   
  for any $\bvarphi\in [P_r(T)]$.  This  completes the proof of this lemma.
  \end{proof}

Let  $u$ and $u_h \in V_{h}^0$ be the exact solution of the Poisson equation \eqref{model} and its numerical approximation arising from  WG Algorithm  \ref{PDWG1}. We define the error function, denoted by $e_h$, as follows
\begin{equation}\label{error} 
e_h=u-u_h.
\end{equation}

\begin{lemma}\label{errorequa}
The error function $e_h$ given in (\ref{error}) satisfies the following error equation;  i.e., 
\begin{equation}\label{erroreqn}
(\nabla_w e_h, \nabla_w v)=\ell (u, v), \qquad \forall v\in V_h^0,
\end{equation}
where 
$$
\ell (u, v)=\sum_{T\in {\cal T}_h}\langle (I- Q_r)\nabla u\cdot\bn, v_0-v_b\rangle_{\partial T}.
$$
\end{lemma}
\begin{proof}   Using \eqref{pro}, and letting $\bvarphi= Q_r \nabla u$ in \eqref{2.4new}, gives
\begin{equation}\label{54}
\begin{split}
&\sum_{T\in {\cal T}_h}(\nabla_w u, \nabla_w v)_T\\=&\sum_{T\in {\cal T}_h}(Q_r \nabla u, \nabla_w v)_T\\
=&\sum_{T\in {\cal T}_h}(Q_r \nabla u, \nabla v_0)_T+\langle Q_r \nabla u\cdot \bn, v_b-v_0\rangle_{\partial T}\\
=&\sum_{T\in {\cal T}_h}(\nabla u, \nabla v_0)_T +\langle   Q_r\nabla u\cdot \bn, v_b-v_0\rangle_{\partial T}\\
=&\sum_{T\in {\cal T}_h}(f, v_0)_T+\langle \nabla u\cdot\bn, v_0\rangle_{\partial T}  +\langle   Q_r\nabla u\cdot \bn, v_b-v_0\rangle_{\partial T}\\
=&\sum_{T\in {\cal T}_h}(f, v_0)_T+\langle (I- Q_r)\nabla u\cdot\bn, v_0-v_b\rangle_{\partial T},
\end{split}
\end{equation}
where we used \eqref{model}, the usual integration by parts, and the fact that $\sum_{T\in {\cal T}_h} \langle \nabla u\cdot\bn, v_b\rangle_{\partial T}= \langle \nabla u\cdot\bn, v_b\rangle_{\partial \Omega}=0$ since $v_b=0$ on $\partial \Omega$.  

Subtracting \eqref{WG} from \eqref{54}  gives 
$$
(\nabla_w e_h, \nabla_w v)=\sum_{T\in {\cal T}_h}\langle (I- Q_r)\nabla u\cdot\bn, v_0-v_b\rangle_{\partial T}.
$$

This completes the proof of the lemma.
\end{proof}

\section{Error Estimates}

\begin{lemma}\cite{wy3655}
Let ${\cal T}_h$ be a finite element partition of the domain $\Omega$ satisfying the shape regular assumption as specified in \cite{wy3655}. For any $0\leq s \leq 1$, $0\leq n \leq k$ and $0\leq m \leq r$, there holds
\begin{eqnarray}\label{error1}
 \sum_{T\in {\cal T}_h}h_T^{2s}\|\nabla u- Q_r \nabla u\|^2_{s,T}&\leq& C  h^{2m}\|u\|^2_{m+1},\\
\label{error2}
\sum_{T\in {\cal T}_h}h_T^{2s}\|u- Q _0u\|^2_{s,T}&\leq& C h^{2n+2}\|u\|^2_{n+1}.
\end{eqnarray}
 \end{lemma}

\begin{lemma}
Assume  the exact solution $u$ of the Poisson equation \eqref{model} is sufficiently regular such that $u\in H^{k+1} (\Omega)$. There exists a constant $C$, such that the following estimate holds true; i.e.,
\begin{equation}\label{erroresti1}
\3bar u-Q_hu \3bar \leq Ch^k\|u\|_{k+1}.
\end{equation}
\end{lemma}
\begin{proof}
Using \eqref{2.4new}, Cauchy-Schwarz inequality, the trace inequalities \eqref{tracein}-\eqref{trace}, the estimate \eqref{error2} with $n=k$ and $s=0, 1$,  we have  
\begin{equation*}
\begin{split}
&\quad\sum_{T\in {\cal T}_h}(\nabla_w(u-Q_hu), \bv)_T\\
&=\sum_{T\in {\cal T}_h}(\nabla(u-Q_0u),  \bv)_T+\langle Q_0u-Q_bu, \bv\cdot\bn\rangle_{\partial T}\\
&\leq \Big(\sum_{T\in {\cal T}_h}\|\nabla(u-Q_0u)\|^2_T\Big)^{\frac{1}{2}} \Big(\sum_{T\in {\cal T}_h}\|\bv\|_T^2\Big)^{\frac{1}{2}}\\&\quad+ \Big(\sum_{T\in {\cal T}_h} \|Q_0u-Q_bu\|_{\partial T} ^2\Big)^{\frac{1}{2}}\Big(\sum_{T\in {\cal T}_h} \|\bv\|_{\partial T}^2\Big)^{\frac{1}{2}}\\
&\leq\Big(\ \sum_{T\in {\cal T}_h} \|\nabla(u-Q_0u)\|_T^2\Big)^{\frac{1}{2}}\Big(\sum_{T\in {\cal T}_h} \|\bv\|_T^2\Big)^{\frac{1}{2}}\\&\quad+\Big(\sum_{T\in {\cal T}_h}h_T^{-1} \|Q_0u-u\|_{T} ^2+h_T \|Q_0u-u\|_{1,T} ^2\Big)^{\frac{1}{2}}\Big(\sum_{T\in {\cal T}_h}Ch_T^{-1}\|\bv\|_T^2\Big)^{\frac{1}{2}}\\
&\leq Ch^k\|u\|_{k+1}\Big(\sum_{T\in {\cal T}_h} \|\bv\|_T^2\Big)^{\frac{1}{2}},
\end{split}
\end{equation*}
 for any $\bv\in [P_r(T)]$. 
 
Letting $\bv=\nabla_w(u-Q_hu)$ gives 
$$
\sum_{T\in {\cal T}_h}(\nabla_w(u-Q_hu), \nabla_w(u-Q_hu))_T\leq 
 Ch^k\|u\|_{k+1}\3bar u-Q_hu \3bar.$$  
 
 This completes the proof of the lemma.
\end{proof}

\begin{theorem}
Assume  the exact solution $u$ of the Poisson equation \eqref{model} is sufficiently regular such that $u\in H^{k+1} (\Omega)$. There exists a constant $C$, such that the following error estimate holds true; i.e., 
\begin{equation}\label{trinorm}
\3bar u-u_h\3bar \leq Ch^k\|u\|_{k+1}.
\end{equation}
\end{theorem}
\begin{proof}
For   the right-hand side of the error equation \eqref{erroreqn}, using Cauchy-Schwarz inequality, the trace inequality \eqref{tracein},  the estimate \eqref{error1} with $m=k$ and $s=0,1$, and \eqref{normeq}, we have   
\begin{equation}\label{erroreqn1}
\begin{split}
&\Big|\sum_{T\in {\cal T}_h}\langle (I- Q_r)\nabla u\cdot\bn, v_0-v_b\rangle_{\partial T}\Big|\\
\leq & C(\sum_{T\in {\cal T}_h}\|(I- Q_r)\nabla u\cdot\bn\|^2_T+h_T^2\|\nabla((I- Q_r)\nabla u\cdot\bn)\|^2_T)^{\frac{1}{2}} \\&\cdot(\sum_{T\in {\cal T}_h}h_T^{-1}\|v_0-v_b\|^2_{\partial T})^{\frac{1}{2}}\\
\leq & Ch^k\|u\|_{k+1} \| v\|_{1,h}\\
\leq & Ch^k\|u\|_{k+1} \3bar v\3bar.
\end{split}
\end{equation}
Substituting \eqref{erroreqn1}  into \eqref{erroreqn}  gives
\begin{equation}\label{err}
(\nabla_w e_h, \nabla_w v)\leq   Ch^k\|u\|_{k+1} \3bar  v\3bar.
\end{equation} 

Using Cauchy-Schwarz inequality, letting $v=Q_hu-u_h$ in \eqref{err}, the  estimate \eqref{erroresti1},  gives
\begin{equation*}
\begin{split}
& \3bar u-u_h\3bar^2\\=&\sum_{T\in {\cal T}_h}(\nabla_w (u-u_h), \nabla_w (u-Q_hu))_T+(\nabla_w (u-u_h), \nabla_w (Q_hu-u_h))_T\\
\leq &\Big(\sum_{T\in {\cal T}_h}\|\nabla_w (u-u_h)\|^2_T\Big)^{\frac{1}{2}} \Big(\sum_{T\in {\cal T}_h}\| \nabla_w (u-Q_hu)\|^2_T\Big)^{\frac{1}{2}} \\& +\sum_{T\in {\cal T}_h} (\nabla_w (u-u_h), \nabla_w (Q_hu-u_h))_T\\
\leq &\3bar u-u_h \3bar  \3bar u-Q_hu \3bar+ Ch^k\|u\|_{k+1} \Big(\sum_{T\in {\cal T}_h}\|\nabla_w (Q_hu-u_h)\|^2_T\Big)^{\frac{1}{2}} \\
\leq &\3bar u-u_h  \3bar  h^k\|u\|_{k+1} + Ch^k\|u\|_{k+1} \Big(\sum_{T\in {\cal T}_h} (\|\nabla_w (Q_hu-u)\|^2_T+\|\nabla_w (u-u_h)\|^2_T\Big)^{\frac{1}{2}} \\
\leq &\3bar u-u_h  \3bar  h^k\|u\|_{k+1} + Ch^k\|u\|_{k+1}   h^k\|u\|_{k+1}+Ch^k\|u\|_{k+1} \3bar u-u_h\3bar.
\end{split}
\end{equation*}
This gives
\begin{equation*}
\begin{split}
 \3bar u-u_h\3bar  \leq Ch^k\|u\|_{k+1}.
\end{split}
\end{equation*} 

This completes the proof of the theorem.
\end{proof}

\section{Error Estimates in $L^2$ Norm}

The standard duality argument is employed to obtain the $L^2$ error estimate. Recall that $e_h=u-u_h=\{e_0, e_b\}$.  We further introduce $  \zeta_h =Q_hu - u_h=\{  \zeta_0,   \zeta_b\}$. 

The dual problem for the Poisson equation \eqref{model}  seeks $w \in H_0^1(\Omega)$ satisfying 
\begin{equation}\label{dual}
-\Delta w=  \zeta_0, \qquad \text{in}\ \Omega.
\end{equation}
Assume that the $H^2$-regularity property for the dual problem \eqref{dual} holds true; i.e.,
 \begin{equation}\label{regu2}
 \|w\|_2\leq C\|  \zeta_0\|.
 \end{equation}
 
 \begin{theorem}
Assume  the exact solution $u$ of the Poisson equation \eqref{model} is sufficiently regular such that $u\in H^{k+1} (\Omega)$.  Let $u_h\in V_h^0$ be the numerical solution of the WG Algorithm \ref{PDWG1}. Assume that  the $H^2$-regularity assumption  \eqref{regu2}  holds true.  There exists a constant $C$ such that 
\begin{equation*}
\|e_0\|\leq Ch^{k+1}\|u\|_{k+1}.
\end{equation*}
 \end{theorem}
 
 \begin{proof}
 Testing \eqref{dual} by $  \zeta_0$, using the usual integration by parts, the fact that $\sum_{T\in {\cal T}_h} \langle \nabla w\cdot\bn,   \zeta_b\rangle_{\partial T}=\langle \nabla w\cdot\bn,   \zeta_b\rangle_{\partial \Omega}=0$ since  $  \zeta_b=Q_bu-u_b=0$ on $\partial\Omega$, we obtain
 \begin{equation}\label{e1}
 \begin{split}
 \|  \zeta_0\|^2 =-(\Delta w,   \zeta_0) 
  = \sum_{T\in {\cal T}_h}(\nabla w, \nabla   \zeta_0)_T-\langle \nabla w\cdot\bn,   \zeta_0-  \zeta_b \rangle_{\partial T}.
 \end{split}
 \end{equation}
 Letting  $u=w$ and $v=  \zeta_h$ in \eqref{54} gives
$$ \sum_{T\in {\cal T}_h}(\nabla_w w, \nabla_w   \zeta_h)_T\\  = \sum_{T\in {\cal T}_h}(\nabla w, \nabla   \zeta_0)_T+\langle   Q_r\nabla w\cdot \bn,   \zeta_b-  \zeta_0\rangle_{\partial T},
$$
which is equivalent to 
$$
\sum_{T\in {\cal T}_h}(\nabla w, \nabla   \zeta_0)_T=\sum_{T\in {\cal T}_h}(\nabla_w w, \nabla_w   \zeta_h)_T-\langle   Q_r\nabla w\cdot \bn,   \zeta_b-  \zeta_0\rangle_{\partial T}.
$$
Substituting the above equation into \eqref{e1} and using \eqref{erroreqn} gives
\begin{equation}\label{e2}
 \begin{split}
& \|  \zeta_0\|^2  \\
  = &\sum_{T\in {\cal T}_h}(\nabla_w w, \nabla_w   \zeta_h)_T-\langle   Q_r\nabla w\cdot \bn,   \zeta_b-  \zeta_0\rangle_{\partial T} -\langle \nabla w\cdot\bn,   \zeta_0-  \zeta_b \rangle_{\partial T}\\
  =& \sum_{T\in {\cal T}_h}(\nabla_w w, \nabla_w e_h)_T+(\nabla_w w, \nabla_w (Q_hu-u))_T-\ell(w,   \zeta_h)\\ 
  =& \sum_{T\in {\cal T}_h}(\nabla_w Q_hw, \nabla_w e_h)_T+(\nabla_w (w-Q_hw), \nabla_w e_h)_T\\&+(\nabla_w w, \nabla_w (Q_hu-u))_T-\ell(w,   \zeta_h)\\ 
  =&\ell(u, Q_hw) + \sum_{T\in {\cal T}_h}(\nabla_w (w-Q_hw), \nabla_w e_h)_T\\&+(\nabla_w w, \nabla_w (Q_hu-u))_T-\ell(w,   \zeta_h)\\
  =&J_1+J_2+J_3+J_4.
 \end{split}
 \end{equation}
 
We will estimate the four terms $J_i$ for $i=1,\cdots, 4$ on the last line of \eqref{e2} one by one.

Regarding to $J_1$, using Cauchy-Schwarz inequality, the trace inequality \eqref{tracein}, the inverse inequality,  the estimate \eqref{error1} with $m=k$ and $s=0, 1$, the estimate \eqref{error2} with $n=1$, gives
\begin{equation}\label{ee1}
\begin{split}
J_1=&\ell(u, Q_hw)\\
\leq &\Big|\sum_{T\in {\cal T}_h}\langle (I- Q_r)\nabla u\cdot\bn, Q_0w-Q_bw\rangle_{\partial T}\Big|\\
\leq& \Big(\sum_{T\in {\cal T}_h}\| (I- Q_r)\nabla u\cdot\bn\|_{\partial T}^2\Big)^{\frac{1}{2}} \Big(\sum_{T\in {\cal T}_h}\|Q_0w-Q_bw\|_{\partial T}^2\Big)^{\frac{1}{2}} \\
\leq& \Big(\sum_{T\in {\cal T}_h}h_T^{-1}\| (I- Q_r)\nabla u\cdot\bn\|_{T}^2+h_T \| (I- Q_r)\nabla u\cdot\bn\|_{1, T}^2\Big)^{\frac{1}{2}}\\& \Big(\sum_{T\in {\cal T}_h}h_T^{-1}\|Q_0w- w\|_{T}^2+h_T\|Q_0w- w\|_{1,T}^2\Big)^{\frac{1}{2}} \\
\leq &Ch^{-1}h^k\|u\|_{k+1}h^2\|w\|_2\\
\leq & Ch^{k+1}\|u\|_{k+1}\|w\|_2.
\end{split}
\end{equation}

For $J_2$, using Cauchy-Schwarz inequality, \eqref{erroresti1} with $k=1$ and \eqref{trinorm} gives
\begin{equation}\label{ee2}
\begin{split}
J_2\leq \3bar w-Q_hw\3bar \3bar e_h\3bar\leq Ch^{k}\|u\|_{k+1}h\|w\|_2\leq Ch^{k+1}\|u\|_{k+1}\|w\|_2.
\end{split}
\end{equation}

  For $J_3$, denote by $Q^0$ a $L^2$ projection onto $[P_0(T)]$. Using \eqref{2.4} gives
  \begin{equation}\label{ee}
 \begin{split}
 & (\nabla_w (Q_hu-u), Q^0\nabla_w w)_T\\ =& -(Q_0u-u, \nabla \cdot Q^0\nabla_w w)_T +\langle Q_bu-u, Q^0\nabla_w w \cdot \bn\rangle_{\partial T}=0.
 \end{split}
 \end{equation}
  Using \eqref{ee}, Cauchy-Schwarz inequality, \eqref{pro} and \eqref{erroresti1}, gives 
  \begin{equation}
  \begin{split}
  J_3\leq &|\sum_{T\in {\cal T}_h}(\nabla_w w, \nabla_w (Q_hu-u))_T|
  \\
  =&|\sum_{T\in {\cal T}_h}(\nabla_w w-Q^0\nabla_ w w, \nabla_w (Q_hu-u))_T|\\
  =&|\sum_{T\in {\cal T}_h}(Q^r\nabla   w-Q^0 Q^r\nabla w, \nabla_w  (Q_hu-u))_T|\\
  \leq & \Big(\sum_{T\in {\cal T}_h}\|Q^r\nabla   w-Q^0 Q^r\nabla w\|_T^2\Big)^{\frac{1}{2}} \3bar Q_hu-u \3bar\\
  \leq & Ch^k\|u\|_{k+1}h\|w\|_2\\
  \leq & Ch^{k+1}\|u\|_{k+1}\|w\|_2.
  \end{split}
  \end{equation}

For $J_4$, using Cauchy-Schwarz inequality, the trace inequality \eqref{tracein},  Lemma \ref{normeqva},   the estimate \eqref{error1} with $m=1$ and $s=0, 1$,  \eqref{erroresti1}, \eqref{trinorm} gives
\begin{equation}\label{ee4}
\begin{split}
J_4=&\ell(w,   \zeta_h)\\
\leq &\Big|\sum_{T\in {\cal T}_h}\langle (I- Q_r)\nabla w\cdot\bn,   \zeta_0-  \zeta_b\rangle_{\partial T}\Big| \\
\leq& \Big(\sum_{T\in {\cal T}_h}\| (I- Q_r)\nabla w\cdot\bn\|_{\partial T}^2\Big)^{\frac{1}{2}} \Big(\sum_{T\in {\cal T}_h}\|  \zeta_0-  \zeta_b\|_{\partial T}^2\Big)^{\frac{1}{2}}   \\
\leq& \Big(\sum_{T\in {\cal T}_h} \| (I- Q_r)\nabla w\cdot\bn\|_{T}^2+h_T^2 \| (I- Q_r)\nabla w\cdot\bn\|_{1, T}^2\Big)^{\frac{1}{2}} \\&\cdot\Big(\sum_{T\in {\cal T}_h}h_T^{-1}\|  \zeta_0-  \zeta_b\|_{\partial T}^2\Big)^{\frac{1}{2}}  \\
\leq& Ch\|w\|_{2} \|   \zeta_h\|_{1,h}  \\
\leq& Ch\|w\|_{2} \3bar   \zeta_h\3bar  \\
\leq &Ch\|w\|_{2} (\3bar u-u_h\3bar+\3bar u-Q_hu\3bar)  \\
\leq & Ch\|w\|_{2}  (h^k\|u\|_{k+1}+h^{k}\|u\|_{k+1})\\
\leq &Ch^{k+1}\|w\|_2\|u\|_{k+1}.
\end{split}
\end{equation}

  Substituting \eqref{ee1}-\eqref{ee4} into \eqref{e2} and using \eqref{regu2} gives
$$
\|  \zeta_0\|^2\leq Ch^{k+1}\|w\|_2\|u\|_{k+1}\leq Ch^{k+1}  \|u\|_{k+1} \|  \zeta_0\|.
$$
This gives
$$
\|  \zeta_0\|\leq Ch^{k+1} \|u\|_{k+1},
$$
which, using triangle inequality and \eqref{error2} with $n=k$, gives
$$
\|e_0\|\leq \|  \zeta_0\|+\|u-Q_0u\|\leq Ch^{k+1}\|u\|_{k+1}. 
$$

This completes the proof of the theorem. 
\end{proof}

\end{document}